\documentclass[11pt]{article}

\usepackage{setspace}

\usepackage{empheq}
\usepackage{amsmath}
\usepackage{amssymb}
\usepackage{amsthm}
\usepackage{enumitem}
\usepackage{titlesec}
\usepackage{graphicx}
\usepackage{caption}
\usepackage{subcaption}

\usepackage{diagbox}

\usepackage[T1]{fontenc}
\usepackage[utf8]{inputenc}

\usepackage{comment}

\usepackage{indentfirst}
\usepackage[top=1.25in,left=1.25in,right=1.25in]{geometry}
\newlength{\widb}
\newlength{\wida}
\newtheorem{thm}{Theorem}[section]
\newtheorem{lemma}[thm]{Lemma}

\numberwithin{equation}{section}

\author{\Large{Phoebe Hollowbread-Smith and Riccardo W. Maffucci}}
\newcommand{\Addresses}{{
		\footnotesize
		{
		P.~Hollowbread-Smith, \textsc{University of Coventry, United Kingdom CV1}\par\nopagebreak\vspace{-0.35cm}
		\textit{E-mail address}, P.~Hollowbread-Smith: \texttt{smithp48@uni.coventry.ac.uk}}
		\\
		\vspace{0.25cm}
		\\
		R.W.~Maffucci (corresponding author), \textsc{University of Coventry, United Kingdom CV1}\par\nopagebreak\vspace{-0.35cm}
		\textit{E-mail address}, R.W.~Maffucci: \texttt{riccardowm@hotmail.com}}
}
\title{\Large{\uppercase{\bf Generation of $3$-connected, planar \\line graphs}}}

\date{}

\def\calL{\mathcal{L}}
\def\calM{\mathcal{M}}

\def\calP{\mathcal{P}}

\def\calR{\mathcal{R}}

\newcommand{\T}{\mathcal{T}}

\begin{document}
\titleformat{\section}
  {\Large\scshape}{\thesection}{1em}{}
\titleformat{\subsection}
  {\large\scshape}{\thesubsection}{1em}{}
\maketitle


\begin{abstract}
We classify and construct all line graphs that are $3$-polytopes (planar and $3$-connected). Apart from a few special cases, they are all obtained starting from the medial graphs of cubic (i.e., $3$-regular) $3$-polytopes, by applying two types of graph transformations. This is similar to the generation of other subclasses of $3$-polytopes \cite{brinkmann_generation_2005,hasheminezhad2011recursive}.
\end{abstract}

\section{Introduction}
In this paper, we will work with simple graphs $\Gamma$, having no repeated edges or loops. The vertex and edge sets will be denoted by $V(\Gamma)$ and $E(\Gamma)$. We say that a graph $\Gamma$ is $k$-connected for some $k\geq 1$ if $V(\Gamma)\geq k+1$ and, however we remove $k-1$ or fewer vertices, the resulting graph is connected. The vertex connectivity of $\Gamma$ is the largest $k$ such that $\Gamma$ is $k$-connected.

A graph is the $1$-skeleton (wireframe) of a $3$-polytope (sometimes a.k.a. polyhedron in the literature) if and only if it is planar and $3$-connected. This beautiful theorem of Rademacher-Steinitz relates graph theory, discrete geometry, and topology. Hence two polyhedra are homeomorphic if and only if their graphs are isomorphic. Accordingly, we will say that a graph is a $3$-polytope if it is planar and $3$-connected. Its regions are also called faces. The dual graph of a $3$-polytope is well-defined, and is again a $3$-polytope (in general, the dual graph of a planar graph depends on the planar immersion, and may have repeated edges and loops). For examples of small $3$-polytopes with up to $14$ edges, see e.g. \cite[Appendix A]{maffucci2022polyhedral}.

Tutte's Algorithm \cite[Theorem 6.1]{tutt61} constructs all $3$-polytopes iteratively by their size (number of edges). Given a subclass of $3$-polytopes, a natural question is to classify and/or construct them all. For instance, a special class of $3$-polytopes is the maximal planar graphs (planar and such that, however we add any more edges, the resulting graph is non-planar). These are simply the triangulations of the sphere (or of the plane, allowing for an `external' unbounded region). Their duals are the cubic (i.e. $3$-regular) $3$-polytopes. All triangulations of the sphere were constructed in \cite{bowen1967generations}, by (iteratively) applying three transformations to a finite number of starting graphs.


Recently, all quadrangulations \cite{brinkmann_generation_2005} and pentangulations \cite{hasheminezhad2011recursive} of the sphere have been constructed. These are the planar graphs of connectivity at least $2$, duals of the quartic ($4$-regular) and quintic ($5$-regular) simple planar graphs respectively. The subclass of $3$-polytopes where each face is a quadrangle, and the subclass of $3$-polytopes where each face is a pentagon have also been constructed \cite{brinkmann_generation_2005,hasheminezhad2011recursive}. In both cases, the authors applied a finite number of transformations to an infinite family of starting graphs (the duals of the antiprisms in the case of \cite{brinkmann_generation_2005}).

In \cite{mafkpr}, the second author completely classified the subclass of $3$-polytopes that are the Kronecker (a.k.a. tensor, direct) product of graphs. Perhaps surprisingly, one of the factors may be non-planar. Which $3$-polytopes are Cartesian and strong graph products is also known \cite{behmah,jhaslu}. A modification of Tutte's Algorithm allows to generate the class of $3$-polytopes of graph radius $1$ \cite{mafp05}. Several other interesting classes of $3$-polytopes were constructed in \cite{dillencourt1996polyhedra}.

In this paper, we ask the natural question, which $3$-polytopes are line graphs? If $\Gamma$ is a graph, its line graph $\calL(\Gamma)$ is defined by
\[V(\calL(\Gamma))=E(\Gamma)\]
and
\[E(\calL(\Gamma))=\{(e,e') : e,e' \text{ are incident edges in $\Gamma$}\}.\]
If $ab\in E(\Gamma)$, we may denote the corresponding vertex of $\calL(\Gamma)$ simply by $ab$ when there is no risk of confusion.

Line graphs are a very well-studied subclass of graphs. For instance, we know that $\Gamma$ is the line graph of some graph if and only if $\Gamma$ does not contain any of nine forbidden induced subgraphs \cite[Theorem 8.4 and Figure 8.3]{harary}. Problems about the connectivity of line graphs have attracted recent attention \cite{knor2003connectivity,xu2005super,shao2010connectivity,huang2011spanning,toth2012connectivity,shao2018essential}.

The result of our investigation completely characterises and constructs the graphs that are line graphs and also $3$-polytopes. It will be proven in section \ref{sec:2}.
\begin{thm}
	\label{thm:1}
Let $\calP=\calL(G)$ be the line graph of a graph $G$, with $G$ different from the seven in the top row of Figure \ref{fig:sp}. Then $\calP$ is a $3$-polytope  if and only if $G$ is obtained from a cubic $3$-polytope by subdividing each edge at most once, and then adding at most one degree $1$ neighbour to each vertex of degree $3$.
\end{thm}

An example where $G$ is itself a cubic $3$-polytope is depicted in Figure \ref{fig:ex1}. Another example is given in Figure \ref{fig:ex2}. The operation of subdividing an edge with both endpoints of degree $3$ is $\T_1$ in Figure \ref{fig:1}. The operation of adding a degree $1$ neighbour to a vertex of degree $3$ is $\T_2$ in Figure \ref{fig:2}.


\begin{figure}[h!]
	\centering
	\includegraphics[width=3.25cm,clip=false]{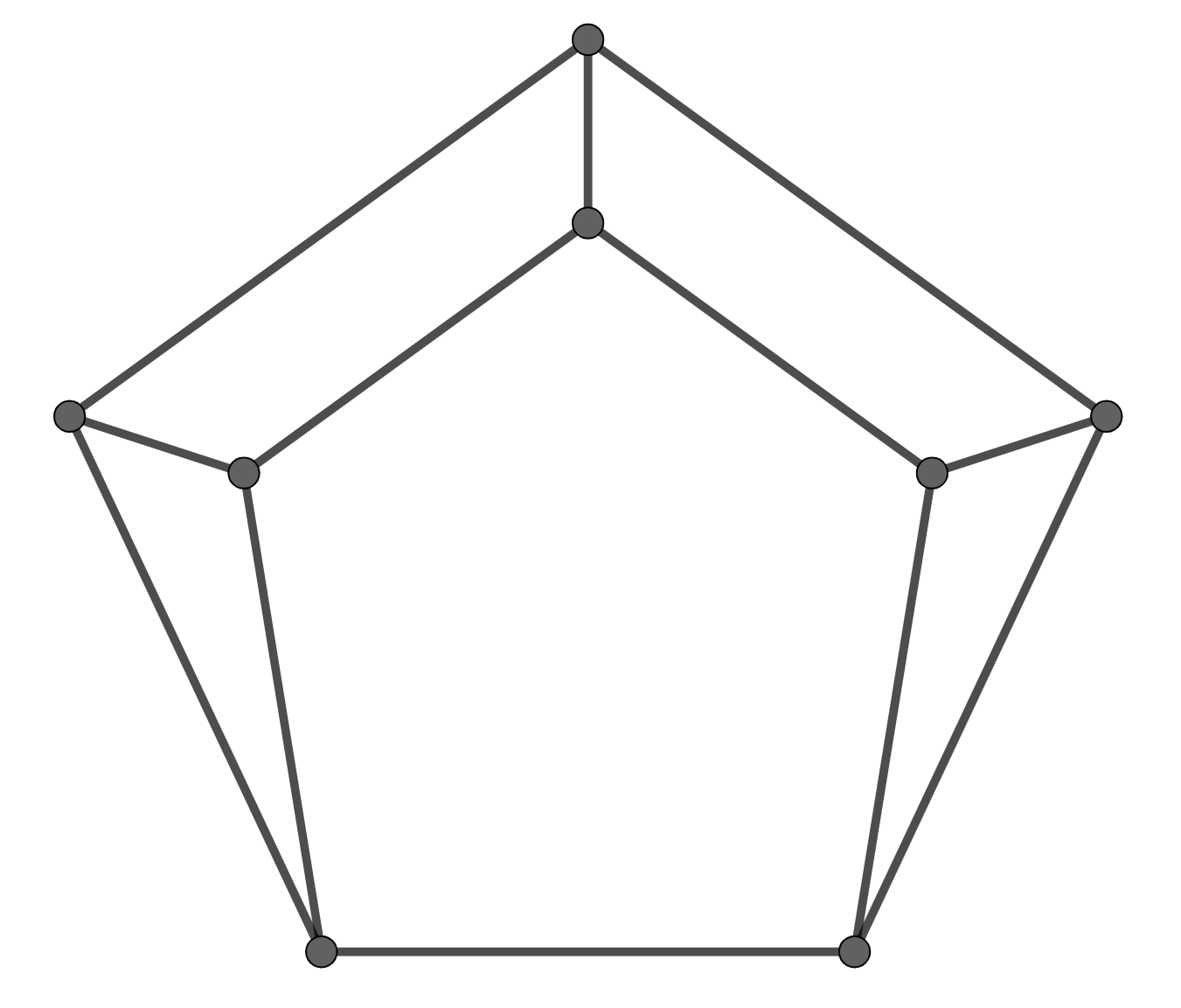}
	\hspace{2cm}
	\includegraphics[width=3.25cm,clip=false]{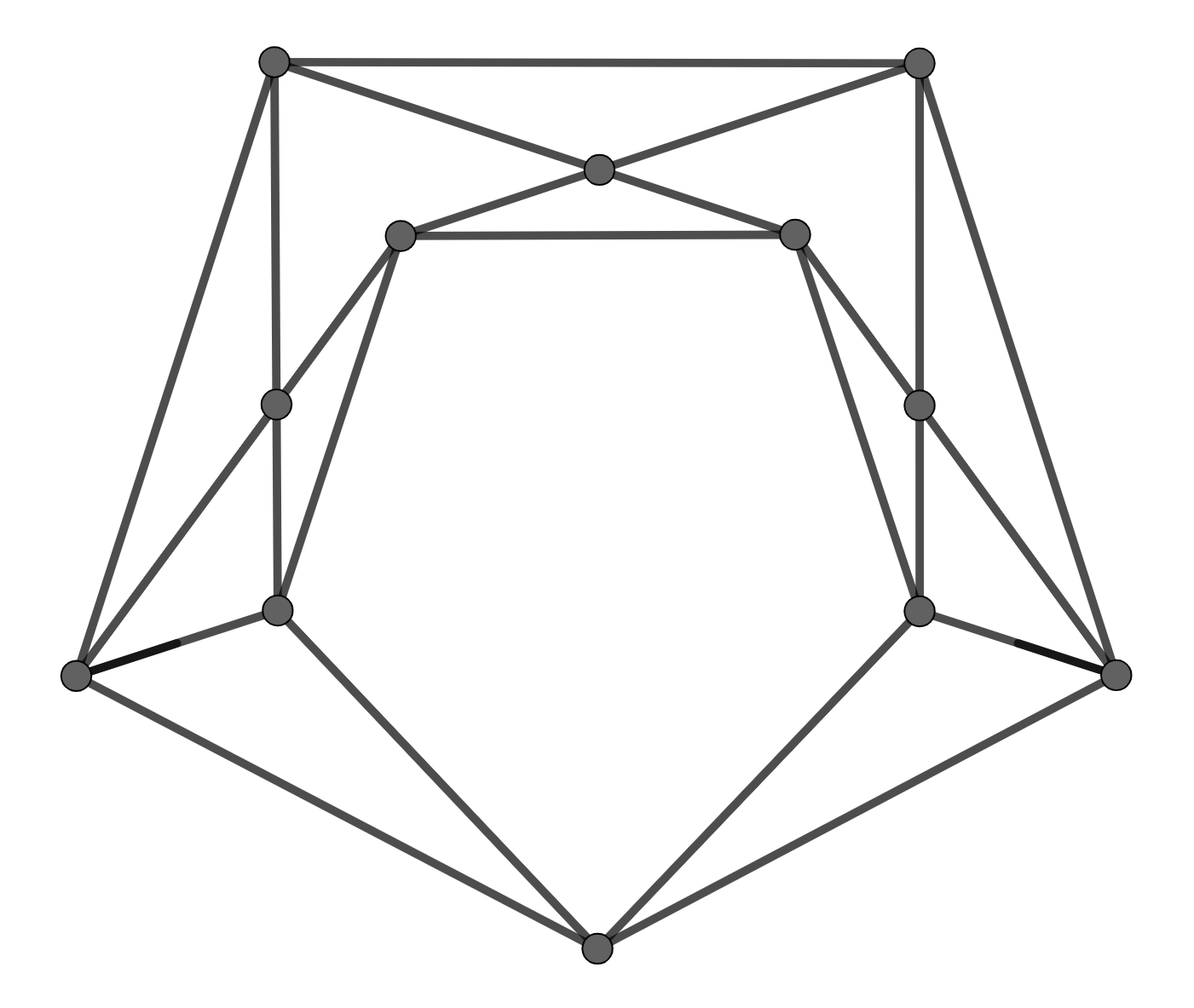}
	\caption{A cubic $3$-polytope $G$ (left), and its $3$-polytopal line graph (right).}
	\label{fig:ex1}
\end{figure}

\begin{figure}[h!]
	\centering
	\includegraphics[width=3.25cm,clip=false]{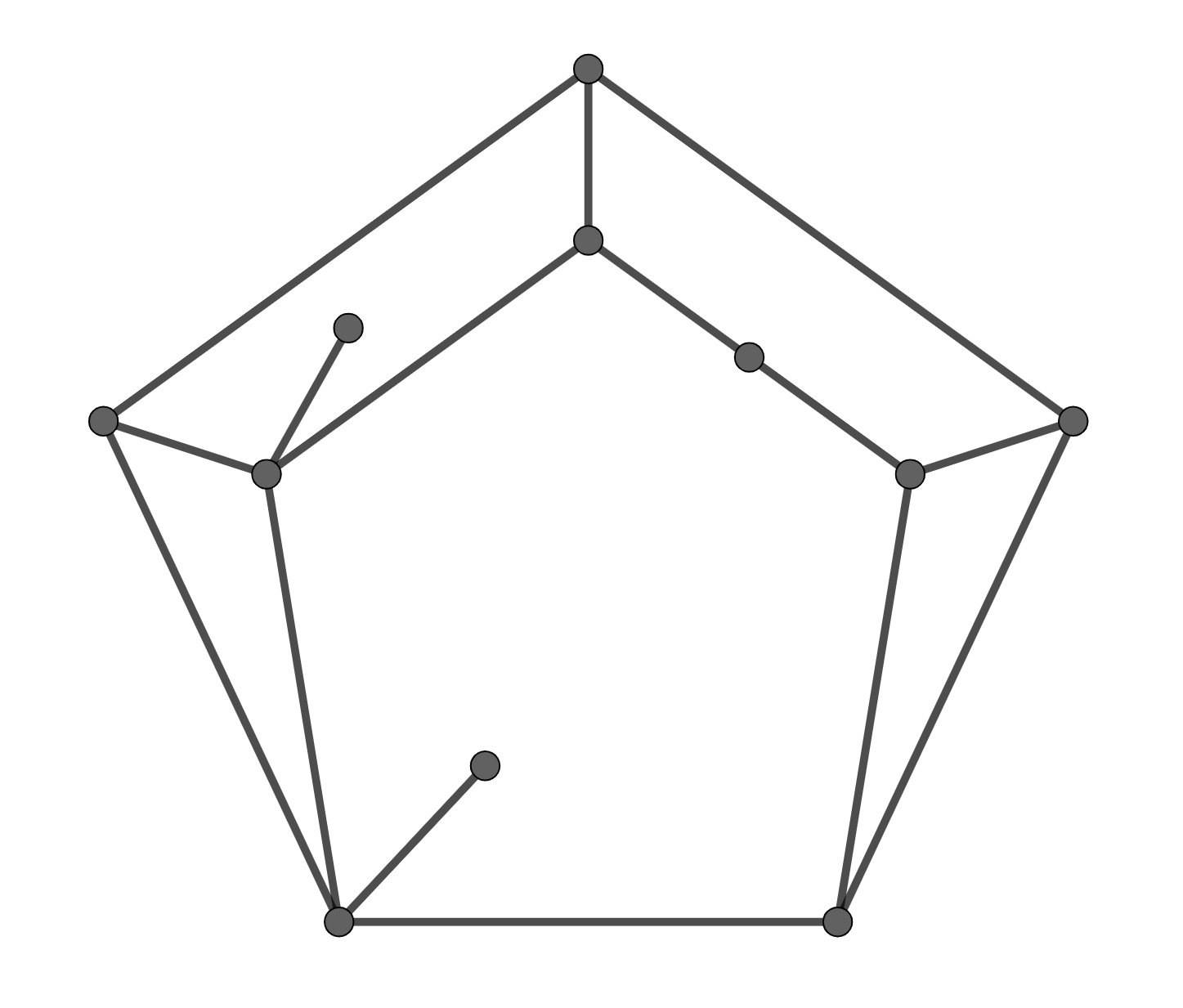}
	\hspace{2cm}
	\includegraphics[width=3.25cm,clip=false]{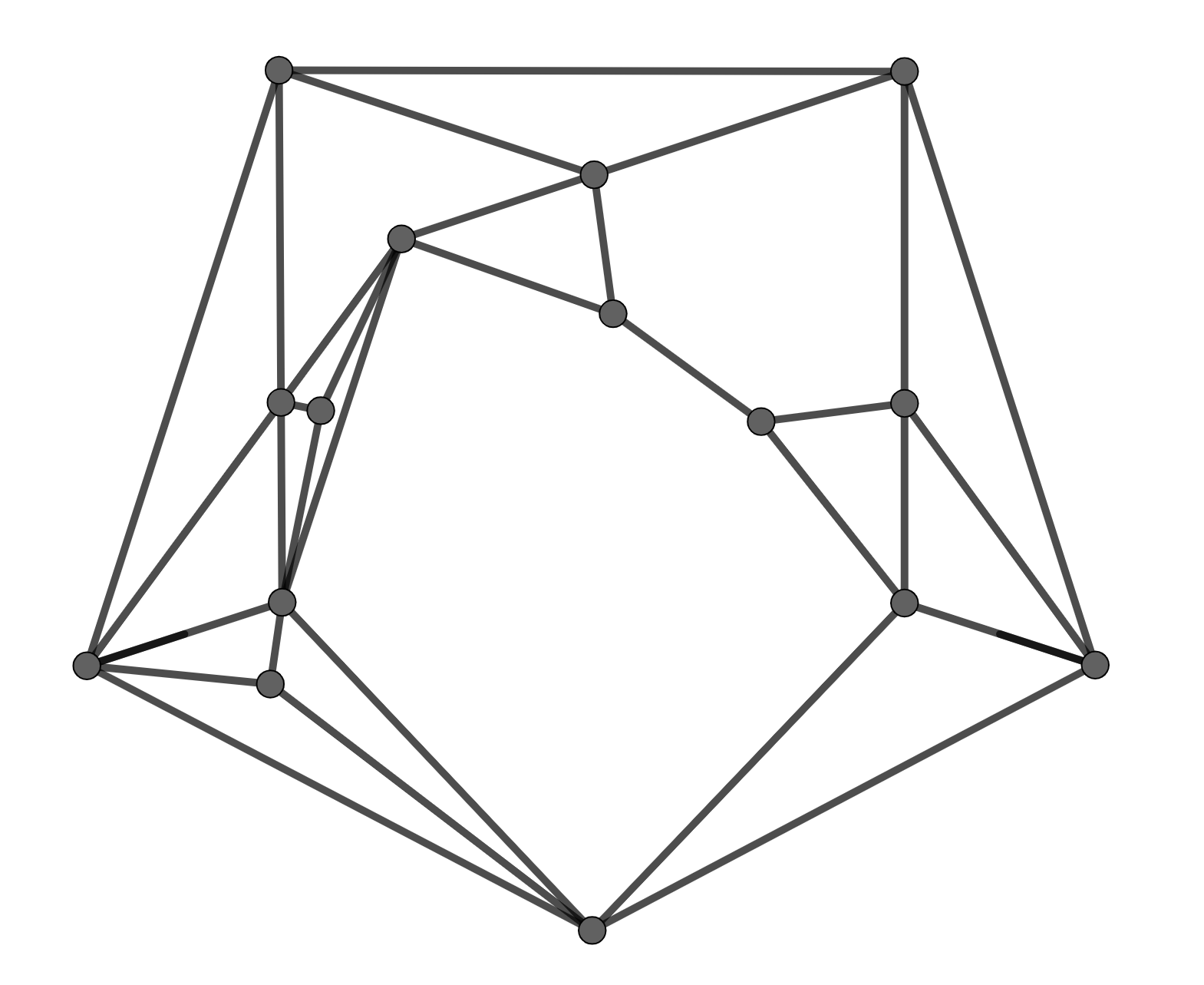}
	\caption{A graph obtained from $G$ of Figure \ref{fig:ex1} via the operations described in Theorem \ref{thm:1} (left), and its $3$-polytopal line graph (right).}
	\label{fig:ex2}
\end{figure}

\begin{figure}[h!]
	\centering
	\[\begin{matrix}
		\includegraphics[width=3cm,clip=false]{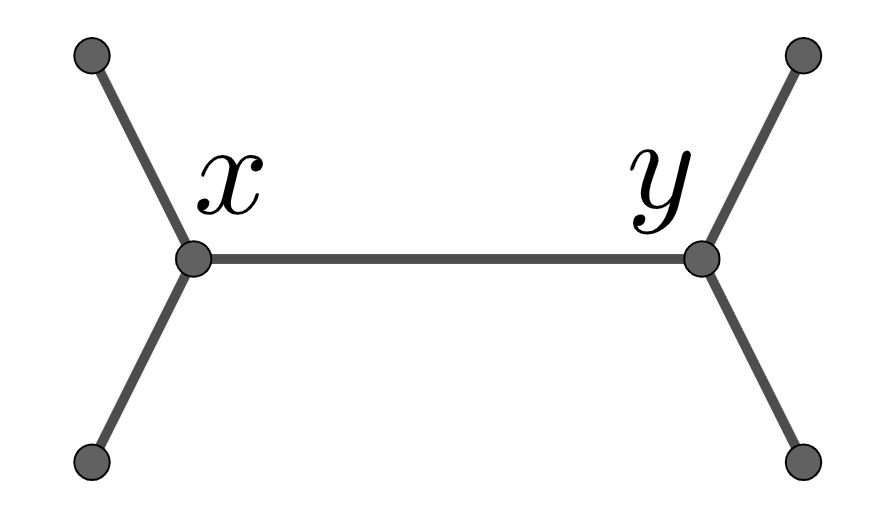}&\longrightarrow_{\T_1}&\includegraphics[width=3cm,clip=false]{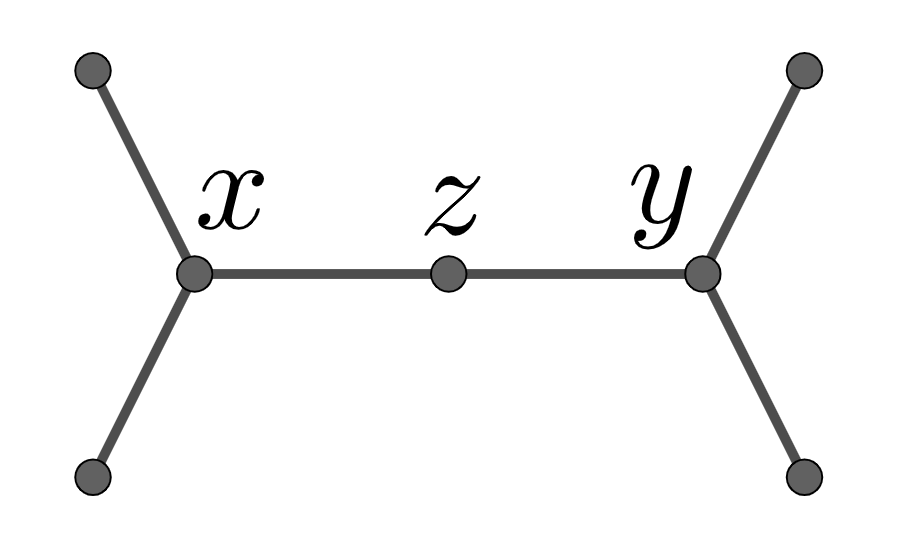}\\\downarrow_{\calL}&&\downarrow_{\calL}\\\includegraphics[width=3.5cm,clip=false]{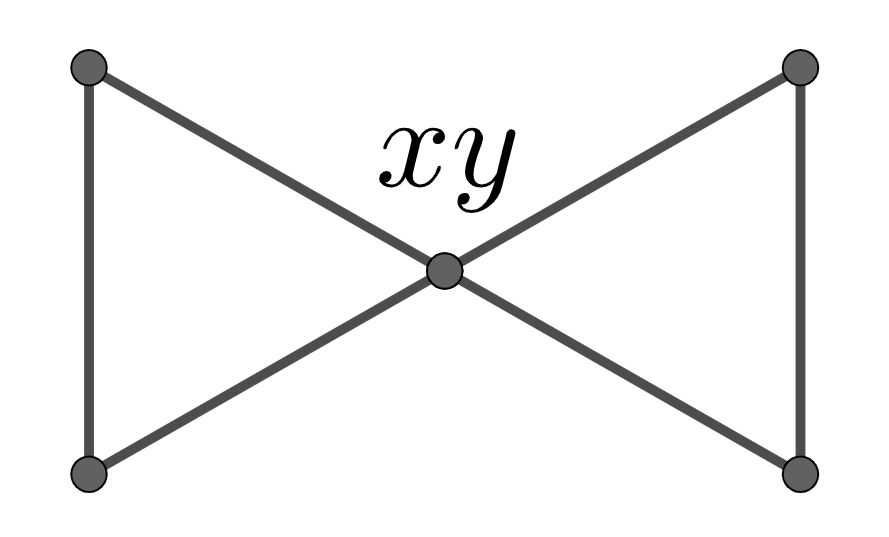}&\longrightarrow_{\T_1'}&\includegraphics[width=3.5cm,clip=false]{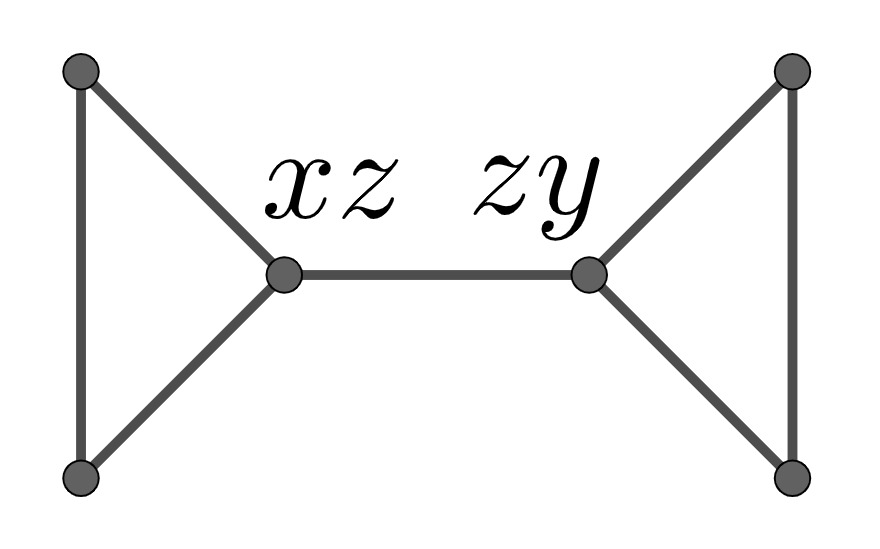}
	\end{matrix}
	\]
	\caption{The transformation $\T_1$ applied to $\Gamma$ subdivides an edge $xy$ with endpoints both of degree $3$. In other words, it deletes $xy$, and adds $z,xz,zy$. This is equivalent to applying $\T_1'$ to the line graph $\calL(\Gamma)$, where $xy$ is a vertex of degree $4$ lying on exactly two triangular faces, which intersect only at $xy$. That is to say, $\calL(\T_1(\Gamma))=\T_1'(\calL(\Gamma))$.}
	\label{fig:1}
\end{figure}

\begin{figure}[h!]
	\centering
\[\begin{matrix}
\includegraphics[width=3.0cm,clip=false]{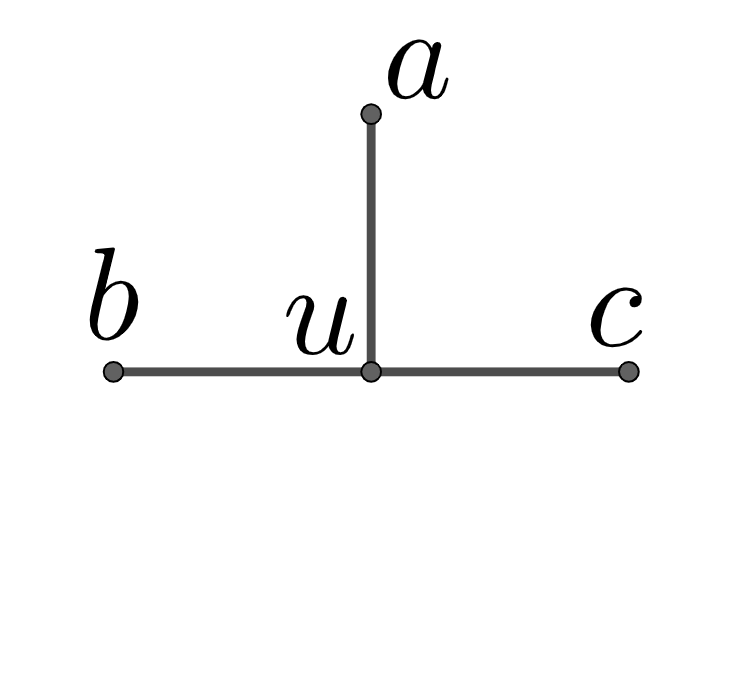}&\longrightarrow_{\T_2}&\includegraphics[width=3.0cm,clip=false]{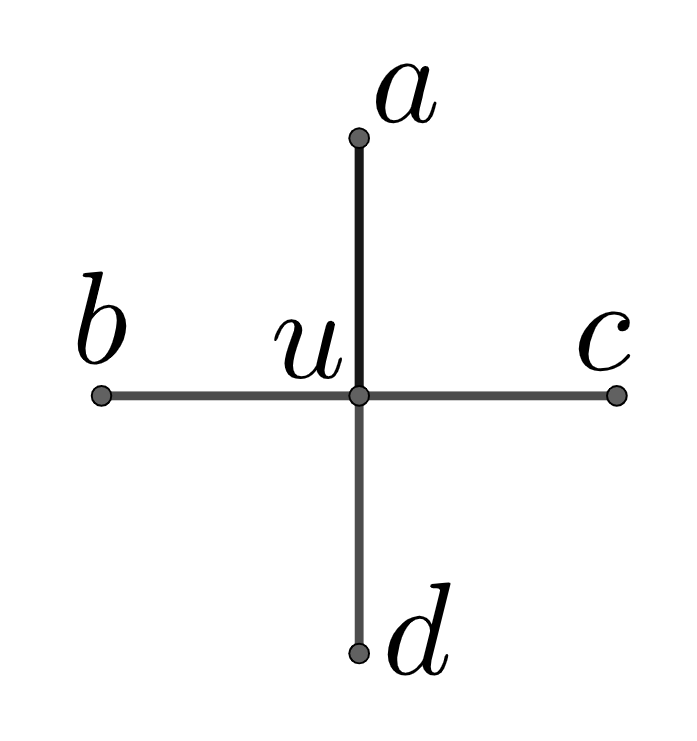}\\\downarrow_{\calL}&&\downarrow_{\calL}\\\includegraphics[width=4.5cm,clip=false]{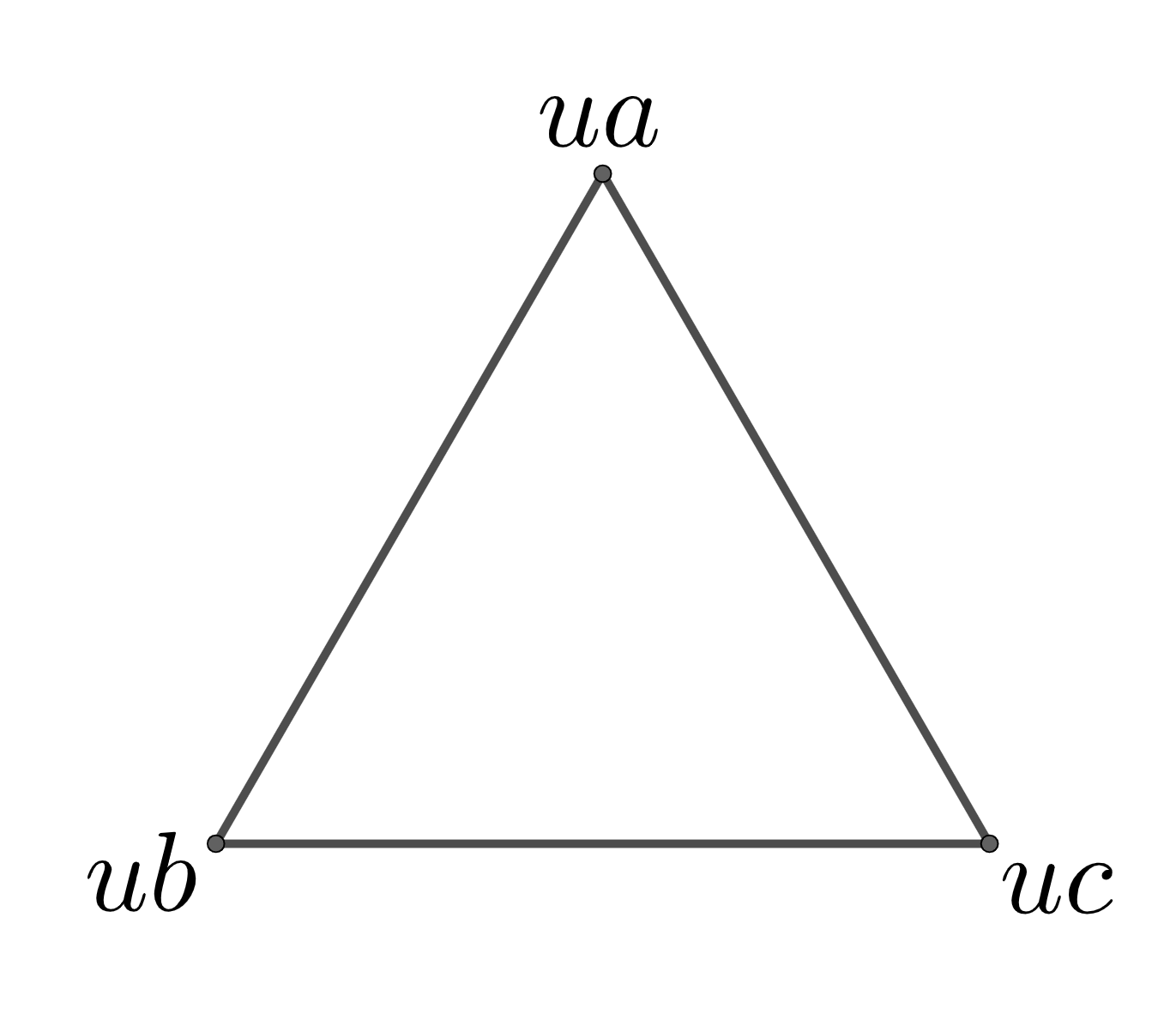}&\longrightarrow_{\T_2'}&\includegraphics[width=4.5cm,clip=false]{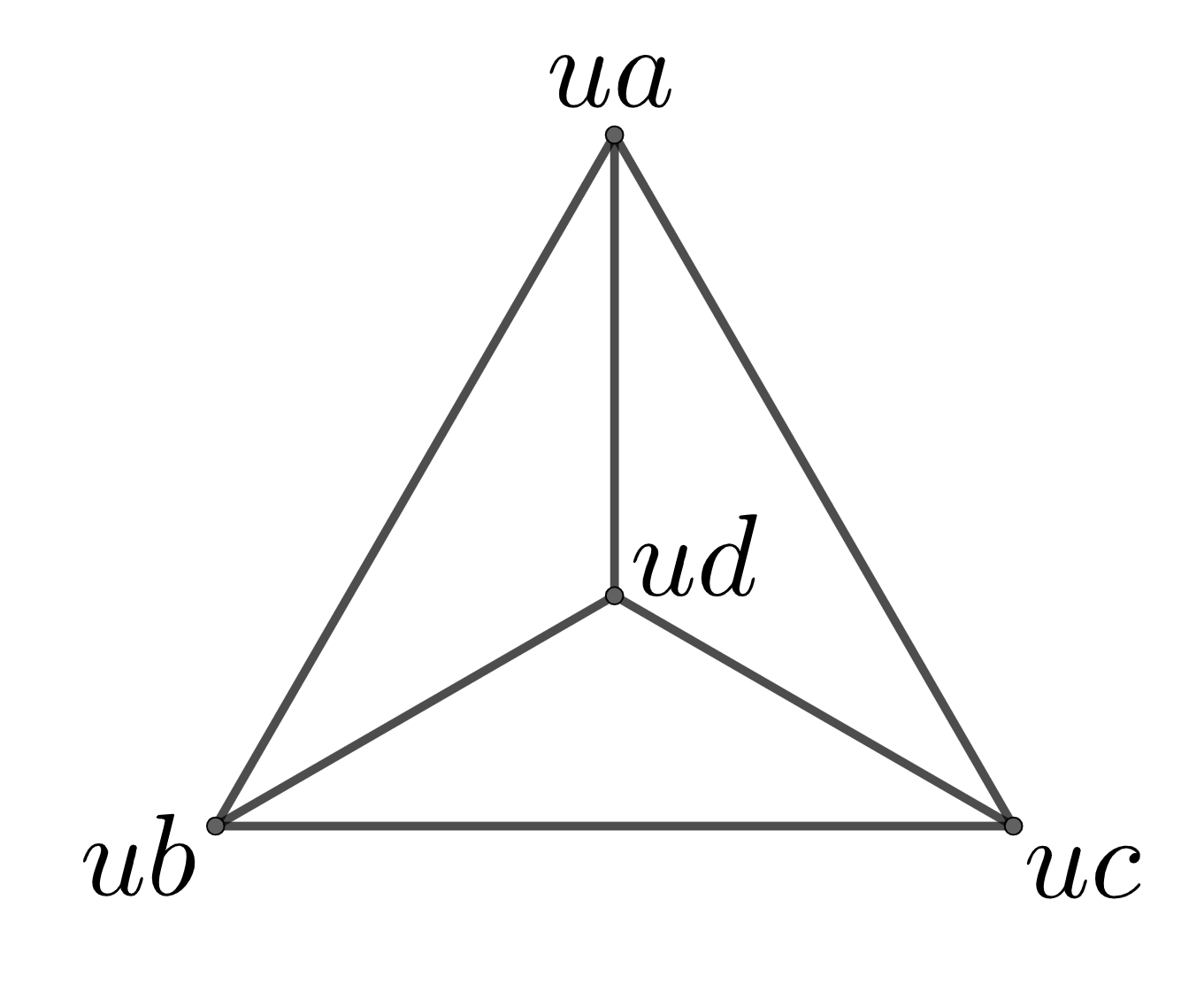}
\end{matrix}
\]
\caption{The transformation $\T_2$ applied to $\Gamma$ adds a vertex $d$ and an edge $ud$, where $\deg_\Gamma(u)=3$. This is equivalent to applying $\T_2'$ to the line graph $\calL(\Gamma)$, i.e., adding a vertex $ud$ and edges $(ud,ua)$, $(ud,ub)$, $(ud,uc)$, where $ua,ub,uc$ is a triangular face in $\calL(\Gamma)$. That is to say, $\calL(\T_2(\Gamma))=\T_2'(\calL(\Gamma))$.}
\label{fig:2}
\end{figure}

\begin{figure}[h!]
	\begin{subfigure}{\widb}
		\centering
		\includegraphics[width=\wida,clip=false]{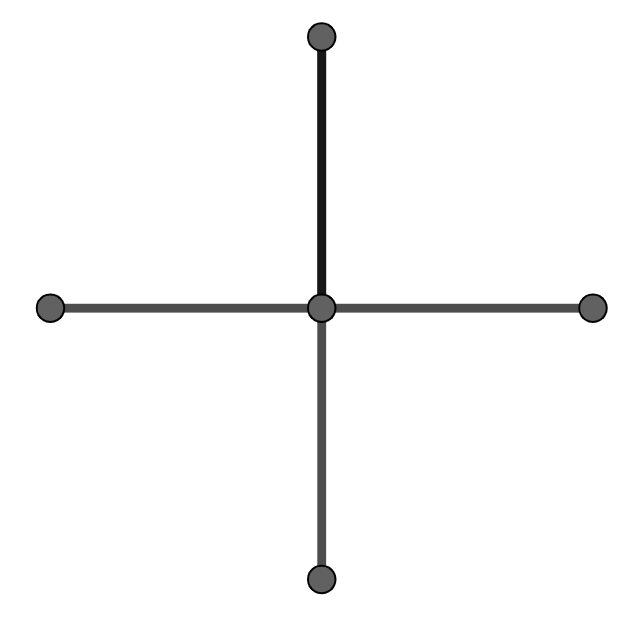}
		\vspace{0.5cm}\\\includegraphics[width=\wida,clip=false]{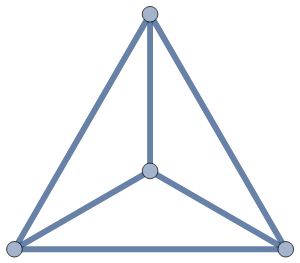}
	\end{subfigure}
	\begin{subfigure}{\widb}
		\centering
		\includegraphics[width=\wida,clip=false]{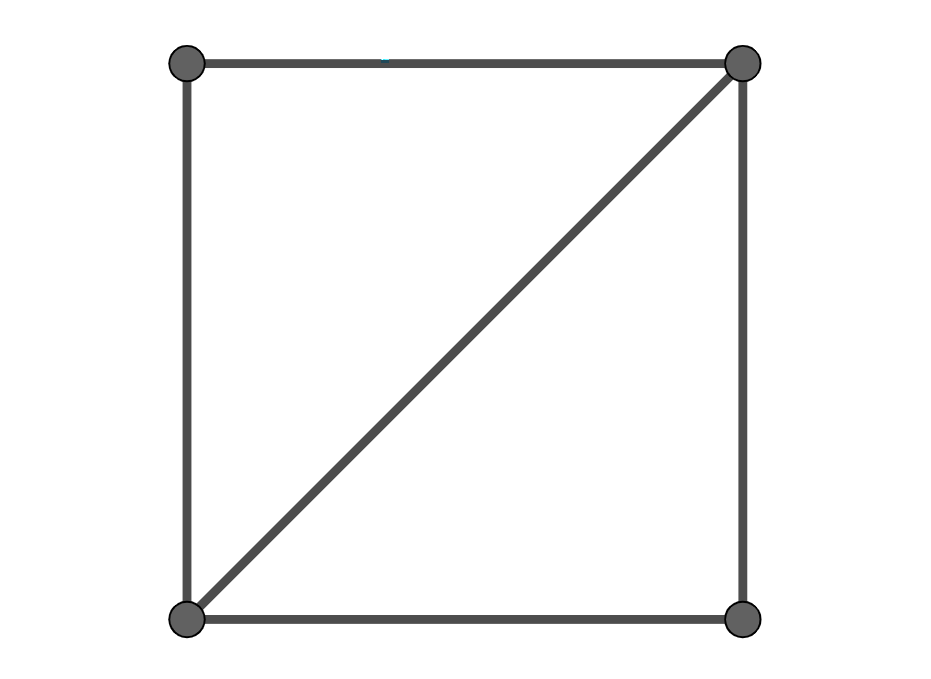}
		\vspace{0.5cm}\\\includegraphics[width=\wida,clip=false]{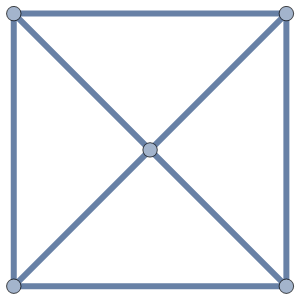}
	\end{subfigure}
	\begin{subfigure}{\widb}
		\centering
		\includegraphics[width=\wida,clip=false]{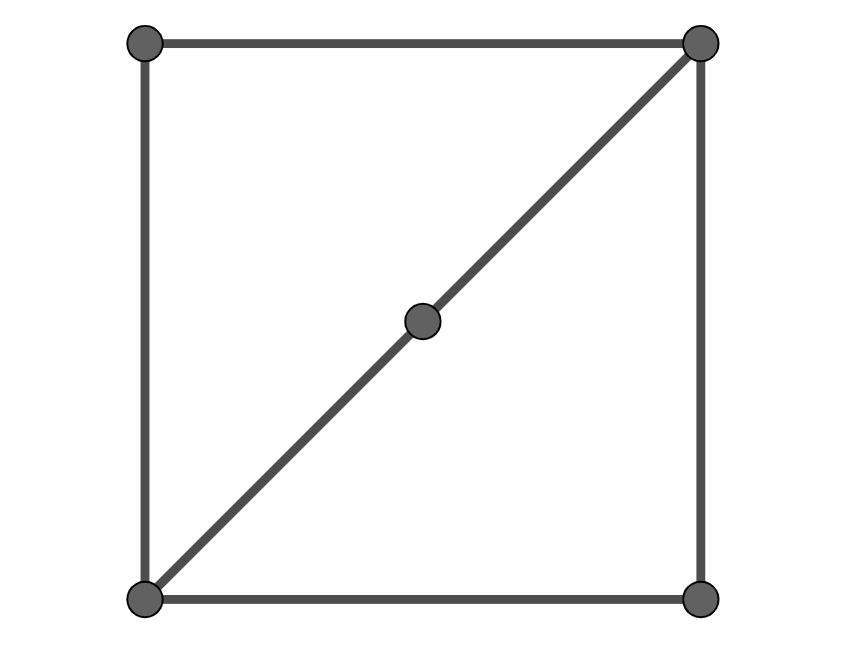}
		\vspace{0.5cm}\\\includegraphics[width=\wida,clip=false]{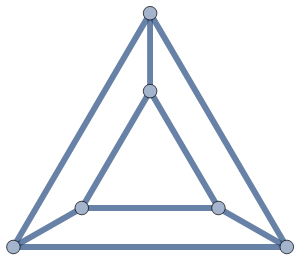}
	\end{subfigure}
	\begin{subfigure}{\widb}
		\centering
		\includegraphics[width=\wida,clip=false]{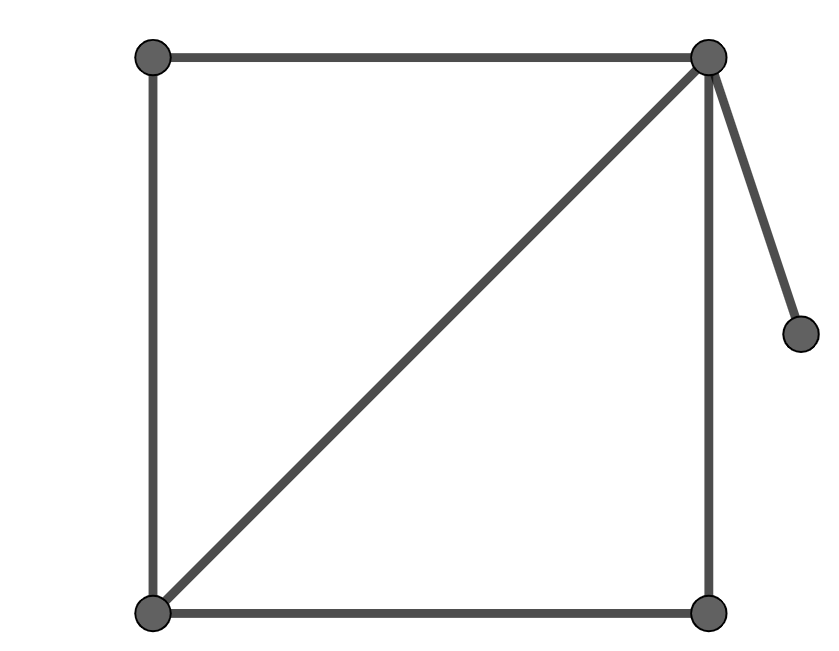}
		\vspace{0.5cm}\\\includegraphics[width=\wida,clip=false]{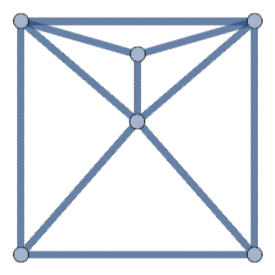}
	\end{subfigure}
	\begin{subfigure}{\widb}
		\centering
		\includegraphics[width=\wida,clip=false]{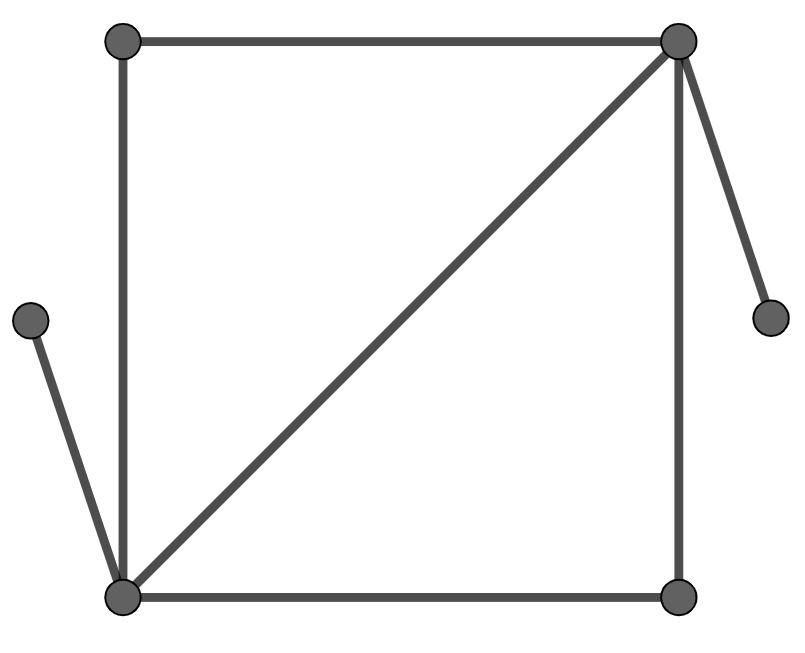}
		\vspace{0.5cm}\\\includegraphics[width=\wida,clip=false]{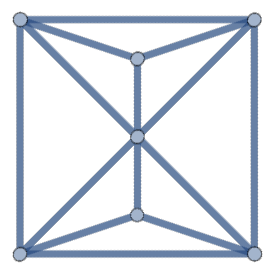}
	\end{subfigure}
	\begin{subfigure}{\widb}
		\centering
		\includegraphics[width=\wida,clip=false]{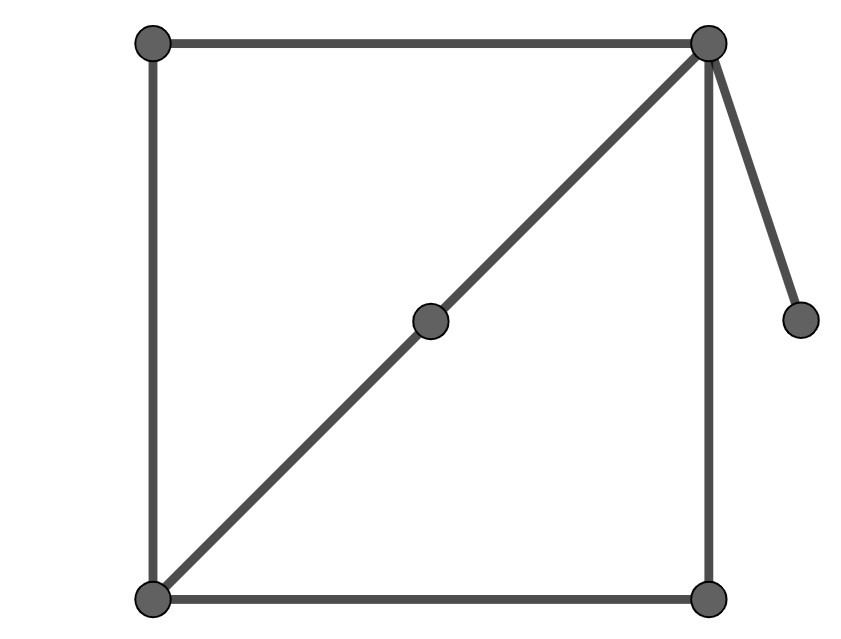}
		\vspace{0.5cm}\\\includegraphics[width=\wida,clip=false]{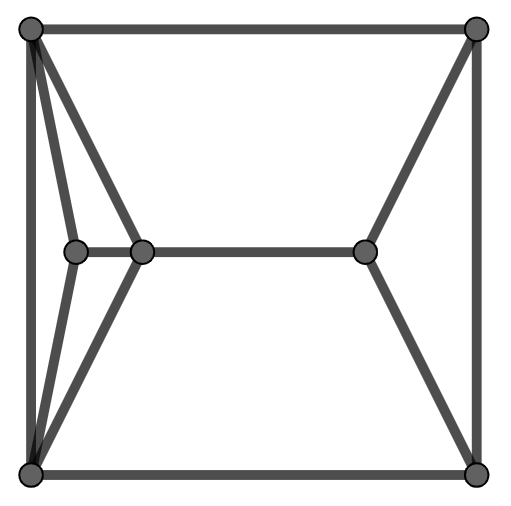}
	\end{subfigure}
	\begin{subfigure}{\widb}
		\centering
		\includegraphics[width=\wida,clip=false]{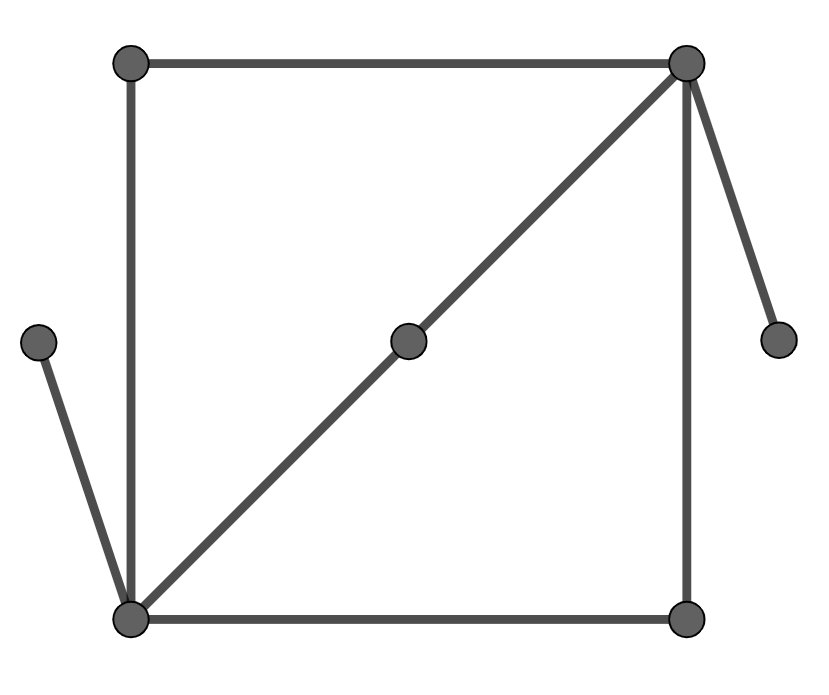}
		\vspace{0.5cm}\\\includegraphics[width=\wida,clip=false]{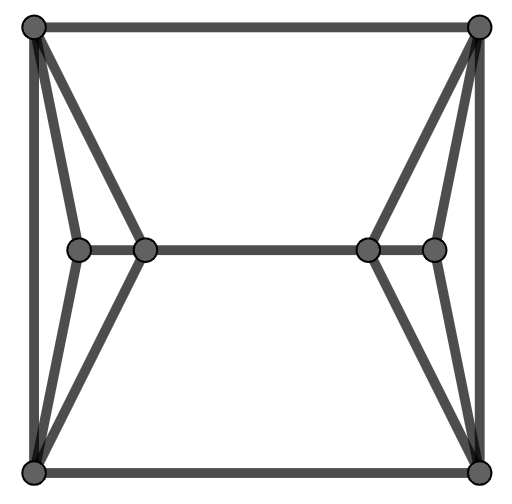}
	\end{subfigure}
	\caption{The seven exceptional graphs $J_i$, $1\leq i\leq 7$ (top row), and their respective $3$-polytopal line graphs.}
	\label{fig:sp}
\end{figure}

A special well-known case of Theorem \ref{thm:1} is given by the cubic $3$-polytopes $\Gamma$, the line graphs of which are quartic $3$-polytopes $\calL(\Gamma)$ (e.g. Figure \ref{fig:ex1}). In this special case, the line graph of $\Gamma$ is also the \textit{medial} graph $\calM(\Gamma)$ of $\Gamma$. For $\Gamma$ a connected, plane graph, the medial graph is defined as
\[V(\calM(\Gamma))=E(\Gamma)\]
and
\[E(\calM(\Gamma))=\{(e,e') : e,e' \text{ are consecutive edges on a region of $\Gamma$}\}.\]
In general, by definition the medial graph of a plane graph $\Gamma$ is a (spanning) subgraph of $\calL(\Gamma)$.

An equivalent reformulation of Theorem \ref{thm:1} is, $\calP$ is a line graph and also a $3$-polytope (other than the seven exceptions in Figure \ref{fig:sp}, bottom row) if and only if $\calP$ is obtained from the medial graph of a cubic $3$-polytope via transformations of type $\T_1'$ followed by transformations of type $\T_2'$ (Figures \ref{fig:1} and \ref{fig:2}). Hence to construct all solutions of our problem, we apply two types of transformations to an infinite starting class of $3$-polytopes.

By the way, for $\Gamma$ a $3$-polytope, $\calM(\Gamma)$ is also the dual of the \textit{radial} graph $\calR(\Gamma)$ \cite[\S 3]{mafpo3}, sometimes called the vertex-face graph. The latter is constructed by inserting a vertex for each vertex and for each face of $\Gamma$, and edges between vertices $(v,F)$ of $\calR(\Gamma)$ whenever in $\Gamma$ the vertex $v$ lies on the face $F$. For instance, the medial and radial graph of the tetrahedron are the octahedron and cube respectively. The medial and radial graphs of a $3$-polytope are $3$-polytopes. The medial is quartic and the radial is a quadrangulation. Dual pairs of $3$-polytopes have the same medial and radial graph.

In the rest of this paper, we will present the proof of Theorem \ref{thm:1}.


\section{Proof of Theorem \ref{thm:1}}
\label{sec:2}
Henceforth we will be using $\Gamma-v$, $\Gamma-v-e$ and similar expressions to mean deleting vertices and/or edges from a graph.

Let $\calP=\calL(G)$ be a $3$-polytope, other than the seven in Figure \ref{fig:sp}, bottom row. In particular, $\calP$ is planar, hence by \cite{sedlacek1964some} (see also \cite[Theorem 1]{sedlavcek1990generalized}), we know that $G$ is planar, of maximal vertex degree $\Delta(G)\leq 4$, and moreover, if
\[\deg_G(u)=4,\]
then $u$ is a separating vertex in $G$ (i.e., $G-u$ is disconnected). Note that the necessary condition $\Delta(G)\leq 4$ is clear: if a vertex in $G$ has degree at least $5$, then $\calL(G)$ contains a copy of $K_5$, hence $\calL(G)$ is non-planar by Kuratowski's Theorem. Further, we may assume that $\delta(G)\geq 1$ and that $G$ is connected. For a characterisation of graphs with planar line graphs in terms of forbidden subgraphs, see \cite{greenwell1972forbidden}.

\begin{lemma}
	\label{lem:14}
Let $\calP=\calL(G)$ be a $3$-polytope. Then all of the following hold.
\begin{enumerate}[label=(\alph*)]
\item
In $G$, two vertices of degree $2$ are never adjacent;
\item
in $G$, every vertex of degree $1$ is adjacent to a vertex of degree $4$;
\item
if $\calP$ is not the tetrahedron, then in $G$ every vertex of degree $4$ is adjacent to exactly one vertex of degree $1$.
\end{enumerate}
\end{lemma}
\begin{proof}
\begin{enumerate}[label=(\alph*)]
\item
By contradiction, let $ab\in E(G)$, with $\deg_{G}(a)=\deg_{G}(b)=2$. Then $\deg_\calP(ab)=2$, contradiction.
\item
Let $\deg_G(v)=1$. Since $\delta(\calP)\geq 3$, then in $G$ each edge is incident to at least three other edges. Therefore, the only neighbour of $v$ has degree at least $4$. Since $\Delta(G)\leq 4$, $v$ is in fact adjacent to a vertex of degree exactly $4$.
\item
Let $\deg_G(u)=4$, with $u$ adjacent to $a,b,c,d$. 
Suppose by contradiction that $\deg_G(a)=\deg_G(b)=1$. Since $\calP-uc-ud$ is connected, it follows that $V(\calP-uc-ud)=\{ua,ub\}$, i.e. $\calP$ is the tetrahedron, contradiction.

Now assume by contradiction that \[\deg_G(a),\deg_G(b),\deg_G(c),\deg_G(d)\geq 2,\]
with $aa',bb',cc',dd'\in E(G)$, and $u\neq a',b',c',d'$. As mentioned above, \cite{sedlacek1964some} implies that $u$ is a separating vertex of $G$, thus the graph
\[\tilde{G}:=G-ua-ub-uc-ud\]
is disconnected. In $\tilde{G}$, w.l.o.g, either $a'$ is in a distinct connected component from all of $b',c',d'$, or $a',b'$ are in the same connected component of $\tilde{G}$, which does not contain $c',d'$. In the first case, $G-ua$ is disconnected, thus $\calP$ has a separating vertex $ua$; in the second case, $G-ua-ub$ is disconnected, thus $\calP$ has a $2$-cut $\{ua,ub\}$. Either way, $\calP$ is not $3$-connected, contradiction.
\end{enumerate}
\end{proof}

We now define
\begin{equation}
G_1:=G-\{\text{vertices of degree } 1\}.
\end{equation}
By Lemma \ref{lem:14}, this is tantamount to applying transformation $\T_2^{-1}$ (the inverse of $\T_2$ in Figure \ref{fig:2}) on all vertices of degree $4$ in $G$. Combining $\Delta(G)\leq 4$ and $\delta(G)\geq 1$ with Lemma \ref{lem:14}, we deduce that in $G_1$ all vertices have degree $2$ or $3$, and moreover, each vertex of degree $2$ in $G_1$ is adjacent only to vertices of degree $3$.


Next, we define
\begin{equation}
	\label{eq:g2}
	G_2:=\text{graph obtained from } G_1 \text{ by contracting all vertices of degree } 2 \text{ not on triangles}.
\end{equation}
That is to say, we have applied transformation $\T_1^{-1}$ (the inverse of $\T_1$ in Figure \ref{fig:1}) to all vertices of degree $2$ in $G_1$ that do not lie on $3$-cycles.

We claim that $G_2$ is a cubic graph. By contradiction, let $u$ be a vertex of degree $2$ in $G_2$, with neighbours $v,w$. As shown above,
\[\deg_{G_2}(v)=\deg_{G_2}(w)=3.\]
By construction, $u$ lies on a $3$-cycle, so that $vw\in E(G_2)$. Let $x$ be the remaining neighbour of $v$, and $y$ the remaining neighbour of $w$. In $G$, the vertex $v$ is either of degree $3$ or $4$, and if $4$, by construction the neighbour of $v$ not in $G_2$ has degree $1$ in $G$. The same obviously holds for $w$. Since
\[\calP-vx-wy\]
is connected, then $x=y$ and $E(G_2)=\{uv,uw,vw,vx,wx\}$, i.e. $G_2$ is the diamond graph. Hence $G_1$ is either the diamond graph or $K_{2,3}$, and in turn, $G$ is one of the graphs $J_i$, $2\leq i\leq 7$ in Figure \ref{fig:sp}, contradiction.

Hence $G_2$ is a cubic, planar (connected) graph. Our next goal is to show that it is $3$-connected, hence it is a $3$-polytope.

\begin{lemma}
The graph $G_2$ in \eqref{eq:g2} is $3$-connected.
\end{lemma}
\begin{proof}
We begin by showing that $G_2$ is $2$-connected. By contradiction, let $u$ be a separating vertex of $G_2$, with neighbours $a,b,c$. W.l.o.g., the component of $G_2-u$ containing $a$ does not contain $b,c$. Thereby, $G_2-ua$ is disconnected. By construction of $G_2$, the edge $ua$ is a bridge in $G$, and both $u,a$ have degree at least $3$ in $G$. Therefore,
\[\calP-ua\]
is disconnected, contradiction. Hence $G_2$ is $2$-connected.

Next, let us check that $G_2$ is in fact $3$-connected. It is well-known that a planar graph is $2$-connected if and only if each of its regions is delimited by a cycle; moreover, a $2$-connected, planar graph is $3$-connected (hence a $3$-polytope) if and only if each pair of regions intersects in either the empty set, or a vertex, or an edge \cite[\S 4]{dieste}. By contradiction, let $R,S$ be regions of $G_2$ such that their intersection contains the vertices $x,y$, with $xy\not\in E(G_2)$. We write
\[R=[r_1,r_2,\dots,r_{i-1},x=r_i,r_{i+1}, \dots,r_{j-1},y=r_j,r_{j+1},\dots,r_m]\]
and
\[S=[s_1,s_2,\dots,s_{k-1},x=s_k,s_{k+1}, \dots,s_{\ell-1},y=s_\ell,s_{\ell+1},\dots,s_n],\]
where $j\geq i+2\geq 4$, and $\ell\geq k+2\geq 4$. Since $G_2$ is cubic, the degree of $r_{i-1},r_{i+1}$ is $3$, thus $r_{i-1},r_{i+1}\not\in S$. Therefore, $x$ is adjacent to the four distinct vertices
\[r_{i-1},r_{i+1},s_{k-1},s_{k+1},\]
contradiction.
\end{proof}

We have succeeded in proving that, if $\calP=\calL(G)$ is a $3$-polytope other than the seven in Figure \ref{fig:sp}, bottom row, then $G$ is obtained from a cubic $3$-polytope $G_2$ by subdividing each edge at most once, and then adding at most one degree $1$ neighbour to each vertex of degree $3$.

On the other hand, let $H_2$ be a cubic $3$-polytope. It is well-known that $\calL(H_2)$ is then a quartic $3$-polytope, that coincides with the medial graph of $H_2$.

We now fix any edge of $H_2$, and apply the transformation $\T_1$ (Figure \ref{fig:1}). We already know that for any graph $\Gamma$,
\[\calL(\T_1(\Gamma))=\T_1'(\calL(\Gamma)).\]
The next goal is to show the following.
\begin{lemma}
	\label{lem:ch}
If $\Gamma$ is a $3$-polytope, then so is $\T_1'(\Gamma)$.
\end{lemma}
\begin{proof}
Clearly $\T_1'$ does not break planarity. To see that it does not break $3$-connectivity, we again use the fact that a $2$-connected, planar graph is $3$-connected if and only if each pair of regions intersects in either the empty set, or a vertex, or an edge \cite[\S 4]{dieste}. If $\T_1'(\Gamma)$ contains regions $R,S$ with intersection containing two non-adjacent vertices, then these are the two regions containing the edge $(xz,zy)$ in Figure \ref{fig:2}, bottom right, otherwise $\Gamma$ would not be $3$-connected either. But then, in Figure \ref{fig:2} bottom left, the top and bottom regions containing $xy$  also share two non-adjacent vertices. Thus $\Gamma$ is not $3$-connected either, contradiction.
\end{proof}

By Lemma \ref{lem:ch}, if $H_1$ is a graph obtained from $H_2$ by subdividing each edge at most once, then $\calL(H_1)$ is a $3$-polytope. Finally, we fix any vertex of degree $3$ in $H_1$ and apply the transformation $\T_2$ (Figure \ref{fig:2}). As we know, for any graph $\Gamma$,
\[\calL(\T_2(\Gamma))=T_2'(\calL(\Gamma)).\]
Applying $\T_2'$ to a $3$-polytope has the effect of adding a vertex inside a triangular face, and edges from it to the three vertices of the triangle. Thus clearly $\T_2'$ preserves planarity and $3$-connectivity.

Therefore, if $H$ is a graph obtained from $H_1$ by applying $\T_2$ at most once to each vertex of degree $3$, then $\calL(H)$ is a $3$-polytope. The proof of Theorem \ref{thm:1} is complete.

\paragraph{Acknowledgements.}
P. H-B. worked on this project as partial fulfilment of her bachelor thesis at Coventry University, autumn 2023-spring 2024, under the supervision of R. M.

\bibliographystyle{abbrv}
\bibliography{biblio}
\Addresses

\end{document}